\documentclass[10pt]{article}
\usepackage{CJK,CJKnumb}
\usepackage{amssymb,a4wide,latexsym,makeidx,epsfig,fleqn}
\usepackage{amsthm}
\usepackage{amsmath}
\usepackage{enumerate}
\newtheorem{theorem}{Theorem}%[section]

\newtheorem{rem}[theorem]{Remark}
\newcommand{\D}{\displaystyle}

\begin{document}
\begin{CJK*}{GBK}{song}
\textwidth 150mm \textheight 225mm
\title{Distance integral complete multipartite graphs with $s=5,6$
\thanks{Supported by the National Natural Science Foundation of China\,(No.11171273)
 and graduate starting seed fund of Northwestern Polytechnical University\,(No.Z2014173).}
 \thanks{Ruosong Yang (1990-), male, native of LuoYang, Henan, an postgraduate student of Northwestern Polytechnical University, engages in graph theory and its application.}}
\author{{Ruosong Yang, Ligong Wang\footnote{Corresponding
author}
}\\
{\small Department of Applied Mathematics, School of Science,
Northwestern Polytechnical University,}\\ {\small  Xi'an, Shaanxi
710072, People's Republic
of China.}\\ {\small E-mail: lnxyangruosong@163.com, lgwangmath@163.com}\\
 }
\date{}
\maketitle
\begin{center}
\begin{minipage}{120mm}
\vskip 0.3cm
\begin{center}
{\small {\bf Abstract}}
\end{center}
{\small Let $D(G)=(d_{ij})_{n\times
n}$ denote the distance matrix of a connected graph $G$ with order
$n$, where $d_{ij}$ is equal to the distance between vertices
$v_{i}$ and $v_{j}$ in $G$. A graph is called distance integral if
all eigenvalues of its distance matrix are integers. In 2014, Yang
and Wang gave a sufficient and necessary condition for complete
$r$-partite graphs $K_{p_{1},p_{2},\ldots,p_{r}}=K_{a_{1}\cdot p_{1},a_{2}\cdot p_{2},\ldots,a_{s}\cdot p_{s}}$ to
be distance integral and obtained such distance integral graphs with $s=1,2,3,4$. However distance integral
complete multipartite graphs $K_{a_{1}\cdot p_{1},a_{2}\cdot
p_{2},\ldots,a_{s}\cdot p_{s}}$ with $s>4$ have not been found. In
this paper, we find and construct some infinite classes of these
distance integral graphs $K_{a_{1}\cdot p_{1},a_{2}\cdot
p_{2},\ldots,a_{s}\cdot p_{s}}$ with $s=5,6$. The problem of the
existence of such distance integral graphs $K_{a_{1}\cdot
p_{1},a_{2}\cdot p_{2},\ldots,a_{s}\cdot p_{s}}$ with arbitrarily
large number $s$ remains open.

 \vskip
0.1in \noindent {\bf 2000 MR Subject Classification}: \ 05C50,
05C12.}
\end{minipage}
\end{center}

\section{Introduction }

Let $G$ be a simple connected undirected graph with vertex set
denoted by $V(G)=\{v_{1},v_{2},\ldots,$ $v_{n}\}$. $A(G)=(a_{ij})$
is an $n\times n$ matrix, where $a_{ij}=1$ if $v_{i}$ and $v_{j}$
are adjacent and $a_{ij}=0$ otherwise. The distance between the
vertices $v_{i}$ and $v_{j}$ is the length of a shortest path
between them, and is denoted by $d_{ij}$. The distance matrix of
$G$, denoted by $D(G)$, is the $n\times n$ matrix whose
$(i,j)$-entry is equal to $d_{ij}$ for $ i,j=1,2,\ldots,n$
 (see \cite{BuHa}). Note that $d_{ii}=0$, $i=1,2,\ldots,n$. The
distance characteristic polynomial (or $D$-polynomial) of $G$ is
$D_{G}(x)=|xI_{n}-D(G)|$,  where $I_{n}$ is the $n\times n$ identity
matrix. The eigenvalues of $D(G)$ are said to be the distance
eigenvalues or $D$-eigenvalues of $G$. Since $D(G)$ is a real
symmetric matrix, the $D$-eigenvalues are real and can be denoted as
$\mu_{1}\geq\mu_{2}\geq\ldots\geq\mu_{n}$. The distance spectral
radius  of $G$ is the largest $D$-eigenvalue $\mu_1$ and denoted by
$\mu(G)$. Assume that $\mu_{1}>\mu_{2}>\ldots>\mu_{t}$ are $t$
distinct $D$-eigenvalues of $G$ with the corresponding
multiplicities $k_{1},k_{2},\ldots,k_{t}$. We denote by $Spec(G)=
 \left(
\begin{array}{lllll}
 \mu_{t} & \mu_{t-1}& \ldots   & \mu_{2}& \mu_{1} \\
   k_{t}  &  k_{t-1} &  \ldots  & k_{2}  & k_{1}\\
\end{array}
\right) $ the Distance spectrum or the $D$-spectrum of $G$. $G$ is called integral if all eigenvalues of $A(G)$ are integers.
Similarly to integral graphs, a graph is called distance integral if
all its $D$-eigenvalues are integers.  Many results about distance
spectral radius and the $D$-eigenvalues of graphs can be found in
\cite{AH,BNP,Das,InGu,LHWS,Mer,StIl,StIn,WaZhou,Zhou,ZhIl}. Some more results about spectral graph theory can be found in
\cite{LZ,T,Zho}.

A {\it complete $r$-partite $(r\geq 2)$ graph}
$K_{p_{1},p_{2},\cdot\cdot\cdot,p_{r}}$ is a graph with a set
$V=V_{1}\cup V_{2}\cup\cdot\cdot\cdot\cup V_{r}$ of
$p_{1}+p_{2}+\cdot\cdot\cdot+p_{r}(=n)$ vertices, where $V_{i}$'s
are nonempty disjoint set, $|V_{i}|=p_{i}$, such that two vertices
in $V$ are adjacent if and only if they belong to different
$V_{i}$'s. Assume that the number of distinct integers of
$p_{1},p_{2},\cdot\cdot\cdot,p_{r}$ is $s$. Without loss of
generality, assume that the first $s$ ones are the distinct integers
such that $p_{1}<p_{2}<\ldots<p_{s}$. Suppose that $a_{i}$ is the
multiplicity of $p_{i}$ for each $i=1,2,\ldots,s.$ The complete
$r$-partite graph
$K_{p_{1},p_{2},\cdot\cdot\cdot,p_{r}}=K_{p_{1},\ldots,
p_{1},\ldots, p_{s},\ldots, p_{s}}$ on $n$ vertices is also denoted
by $K_{a_{1}\cdot p_{1},a_{2}\cdot p_{2},\ldots,a_{s}\cdot p_{s}}$,
where $r=\sum\limits_{i=1}^{s}a_{i}$ and
$n=\sum\limits_{i=1}^{s}a_{i}p_{i}.$ In 2014, Yang and Wang gave a
sufficient and necessary condition for $K_{a_{1}\cdot
p_{1},a_{2}\cdot p_{2},\ldots,a_{s}\cdot p_{s}}$ to be distance
integral and obtained such distance integral graphs $K_{a_{1}\cdot
p_{1},a_{2}\cdot p_{2},\ldots,a_{s}\cdot p_{s}}$ with $s=1,2,3,4$
(see\cite{YW}). In this paper, we find and construct some infinite
classes of these distance integral graphs $K_{a_{1}\cdot
p_{1},a_{2}\cdot p_{2},\ldots,a_{s}\cdot p_{s}}$ with $s=5,6$ by
using a computer search. For $s=5, 6$, we give a positive answer to
a question of Yang et al. \cite{YW}. The problem of the existence of
distance integral complete multipartite graphs $K_{a_{1}\cdot
p_{1},a_{2}\cdot p_{2},\ldots,a_{s}\cdot p_{s}}$ with arbitrarily
large number $s$ remains open.

\section{Preliminaries}
In this section, we shall give some known results on distance
integral complete $r$-partite graphs, which are used in our theorem
proof later.

\noindent\begin{theorem}\label{polynomial} (See Theorem 4.1 of
\cite{LHWS} and \cite{SMHP}) Let $G$ be a complete $r$-partite graph
$K_{p_{1},p_{2},\ldots,p_{r}}$ on $n$ vertices. Then the
$D$-polynomial of $G$  is
\begin{equation}
\label{eq:c-1}\\D_{G}(x)=\prod_{i=1}^r(x+2)^{(p_{i}-1)}\prod_{i=1}^r(x-p_{i}+2)(1-\sum_{i=1}^r\frac{p_{i}}{x-p_{i}+2}).
\end{equation}
\end{theorem}

\noindent \begin{theorem}\label{th:chongyao}(See \cite{YW}) If the complete $r$-partite
graph $K_{p_{1},p_{2},\ldots ,p_{r}}=K_{a_{1}\cdot p_{1},a_{2}\cdot
p_{2},\ldots ,a_{s}\cdot p_{s}}$ on $n$ vertices is distance
integral then there exist integers $\mu_{i}(i=1,2,\ldots ,s)$ such
that $-2 < p_{1}-2 < \mu_{1} < p_{2}-2 < \mu_{2} < \cdots <
p_{s-1}-2 < \mu_{s-1} < p_{s}-2 < \mu_{s} <+\infty$ and the numbers
$a_{1},a_{2},\ldots a_{s}$ defined by
\begin{equation}\label{eq:c-12}
a_{k}=\frac{\prod_{i=1}^{s}(\mu_{i}-p_{k}+2)}{p_{k}\prod_{i=1,i\neq
k}^{s}(p_{i}-p_{k})}, k=1,2,\ldots ,s,
\end{equation}
are positive integers.

 Conversely, suppose that there exist integers $\mu_i$ $(i=1,2,\ldots,s)$
 such that $-2 < p_{1}-2 < \mu_{1} <
p_{2}-2 < \mu_{2} < \cdots < p_{s-1}-2 < \mu_{s-1} < p_{s}-2 <
\mu_{s} <+\infty$ and that the numbers
$a_{k}=\frac{\prod_{i=1}^{s}(\mu_{i}-p_{k}+2)}{p_{k}\prod_{i=1,i\neq
k}^{s}(p_{i}-p_{k})}(k=1,2,\ldots,s)$ are positive integers. Then
the complete $r$-partite graph $K_{p_{1},p_{2},\ldots
,p_{r}}=K_{a_{1}\cdot p_{1},\ldots ,a_{s}\cdot p_{s}}$ is distance
integral.
\end{theorem}

\noindent \begin{theorem}\label{th:newclass} (See Corollary 2.9 of
\cite{YW})For any positive integer
$q$, the complete $r$-partite graph $K_{p_{1}q,p_{2}q,\ldots,p_{r}q}=K_{a_{1}\cdot p_{1}q,a_{2}\cdot
p_{2}q,\ldots,a_{s}\cdot p_{s}q}$ is distance integral if and only
if the complete $r$-partite graph
$K_{p_{1},p_{2},\ldots,p_{r}=K_{a_{1}\cdot p_{1},a_{2}\cdot
p_{2},\ldots,a_{s}\cdot p_{s}}}$ is distance integral.
\end{theorem}

\noindent \begin{rem}\label{re:2-1} Let
$GCD(p_{1},p_{2},\ldots,p_{s})$ denote the greatest common divisor
of the numbers $p_{1},p_{2},\ldots,p_{s}$. We say that a vector
$(p_{1},p_{2},\ldots,p_{s})$ is primitive if
$GCD(p_{1},p_{2},\ldots,p_{s})=1$. Theorem \ref{th:newclass} shows that
it is reasonable to study the complete $r$-partite
graph $K_{p_{1},p_{2},\ldots,p_{r}}=K_{a_{1}\cdot p_{1},a_{2}\cdot
p_{2},\ldots,a_{s}\cdot p_{s}}$ only for primitive
vectors $(p_{1},p_{2},\ldots,p_{s})$.
\end{rem}

\noindent \begin{theorem}\label{th:kuozhan}(See \cite{YW})  Let a complete $r$-partite
graph $K_{p_{1},p_{2},\ldots,p_{r}}=K_{a_{1}\cdot p_{1},a_{2}\cdot
p_{2},\ldots,a_{s}\cdot p_{s}}$ be  distance integral with
eigenvalues $\mu_{i}$. Let $\mu_{i}(\geq0)$ and $p_{i}(>0)
(i=1,2,\ldots,s)$ be integers such that $-2 < p_{1}-2 < \mu_{1}
< p_{2}-2 < \mu_{2} < \cdots < p_{s-1}-2 < \mu_{s-1} < p_{s}-2 <
\mu_{s} <+\infty$  and
\begin{equation}\label{eq:ad-1}a_{k}=\frac{\prod_{i=1}^{s}(\mu_{i}-p_{k}+2)}{p_{k}\prod_{i=1,i\neq
k}^{s}(p_{i}-p_{k})},k=1,2,\ldots,s\end{equation} are positive
integers, then for
\begin{equation}\label{eq:ad-2}b_{k}=\frac{\prod_{i=1}^{s-1}(\mu_{i}-p_{k}+2)(\mu_{s}-p_{k}+2+rt)}{p_{k}\prod_{i=1,i\neq k}^{s}(p_{i}-p_{k})},k=1,2,\ldots,s,\end{equation}
\begin{equation}\label{eq:ad-3}r=LCM(r_{1},r_{2},\ldots,r_{s}), r_{k}=\frac{p_{k}\prod_{i=1,i\neq k}^{s}(p_{i}-p_{k})}{d_{k}},k=1,2,\ldots,s,\end{equation}
\begin{equation}\label{eq:ad-4}d_{k}=GCD(\prod_{i=1}^{s-1}(\mu_{i}-p_{k}+2),p_{k}\prod_{i=1,i\neq
k}^{s}(p_{i}-p_{k})),k=1,2,\ldots,s.\end{equation} the complete
$m$-partite graph $K_{p_{1},p_{2},\ldots,p_{m}}=K_{b_{1}\cdot
p_{1},b_{2}\cdot p_{2},\ldots,b_{s}\cdot p_{s}}$ is distance
integral for every nonnegative integer $t$ with eigenvalues
$\mu_{1},\mu_{2},\ldots,\mu_{s-1},\mu'_{s}=\mu_{s}+rt$. (Similar results for integral complete
multipartite graphs were given in \cite{HP2})
\end{theorem}

\section{Main Results}
\label{sec:ch-inco}

In this section, we construct some distance integral complete $r$-partite graphs with $s=5,6$. Moreover we also get infinitely many new classes of
distance integral complete $r$-partite graphs
$K_{p_{1},p_{2},\ldots,p_{r}}=K_{a_{1}\cdot p_{1},a_{2}\cdot
p_{2},\ldots,a_{s}\cdot p_{s}}$ with $s=5,6$ by using Theorem \ref{th:kuozhan}.(Similar results for integral complete
multipartite graphs were given in \cite{WW})

\noindent \begin{theorem}\label{th:s=5} For $s=5$, integers
$p_{i}(>0)$, $a_{i}(>0)$ and $\mu_{i}(i=1,2,3,4,5)$ are given in
Table 1. $p_{i}$, $a_{i}$ and $\mu_{i}(i=1,2,3,4,5)$ are those of
Theorem \ref{th:chongyao}, where $GCD(p_1,p_2,p_3,p_4,p_5)=1$. Then
for any positive integer $q$ the graph $K_{a_{1}\cdot
p_{1}q,a_{2}\cdot p_{2}q,a_{3}\cdot p_{3}q,a_{4}\cdot
p_{4}q,a_{5}\cdot p_{5}q}$ on $qn(=q\sum_{i=1}^{5}a_{i}p_{i})$
vertices is distance integral.
\end{theorem}
\begin{table}[!htbp]
\centering \caption{Distance integral complete $r$-partite graphs
$K_{a_{1}\cdot p_{1},a_{2}\cdot p_{2},a_{3}\cdot p_{3},a_{4}\cdot
p_{4},a_{5}\cdot
p_{5}}.$} \label{tab:s=5}

\noindent {\footnotesize
$$\begin{tabular}{ccccccccccccccc}
\hline $p_1$&$p_2$&$p_3$&$p_4$&$p_5$&$a_1$& $a_2$&$a_3$&$a_4$&$a_5$&$\mu_1$&$
\mu_2$&$\mu_3$&$\mu_4$&$\mu_5$\\
\hline
                  1 &        3 &        5 &       12 &       20 &     1302 &      254 &      185 &       70 &

132 &        0 &        2 &        7 &       13 &     6478\\

         1 &        3 &        8 &       18 &       26 &     1914 &      520 &      197 &       36 &

46 &        0 &        4 &       14 &       21 &     6901\\

         1 &        3 &        9 &       23 &       35 &     1596 &      444 &       91 &      112 &

22 &        0 &        5 &       13 &       31 &     7105\\

         1 &        3 &       12 &       23 &       37 &      580 &      155 &       77 &       12 &

8 &        0 &        5 &       19 &       32 &     2551\\

         1 &        4 &        8 &       18 &       31 &     3282 &      332 &      101 &       22 &

112 &        1 &        5 &       14 &       20 &     9298\\

         1 &        4 &       13 &       18 &       33 &     2355 &      230 &       82 &       47 &

85 &        1 &        7 &       14 &       23 &     8006\\

         1 &        5 &        8 &       22 &       36 &      754 &       20 &       47 &       24 &

24 &        2 &        4 &       13 &       27 &     2638\\

         1 &        5 &        9 &       19 &       35 &     1099 &      127 &       82 &       32 &

64 &        1 &        5 &       13 &       23 &     5337\\

         1 &        5 &        9 &       22 &       36 &     1688 &       84 &       30 &       57 &

16 &        2 &        6 &       13 &       31 &     4219\\

         1 &        5 &       11 &       19 &       37 &     1933 &      299 &       18 &      116 &

51 &        1 &        8 &       11 &       29 &     7731\\

         1 &        6 &        8 &       10 &       19 &     3112 &      157 &      106 &      136 &

100 &        2 &        5 &        7 &       14 &     8168\\

         1 &        6 &        8 &       21 &       28 &     1915 &      232 &      184 &       21 &

155 &        1 &        5 &       14 &       20 &     9574\\

         1 &        6 &        9 &       16 &       24 &      245 &       19 &        9 &        2 &

4 &        2 &        6 &       13 &       19 &      574\\

         1 &        6 &       10 &       20 &       40 &     2436 &       56 &       74 &       45 &

56 &        3 &        6 &       14 &       28 &     6668\\

         1 &        6 &       12 &       20 &       32 &     1638 &       83 &       98 &       68 &

56 &        2 &        6 &       14 &       25 &     6478\\

         1 &        6 &       15 &       23 &       45 &     4256 &      143 &       82 &       36 &

32 &        3 &       10 &       19 &       37 &     8623\\

         1 &        7 &        9 &       17 &       20 &     1071 &       33 &       31 &       18 &

42 &        3 &        6 &       11 &       16 &     2735\\

         1 &        7 &       11 &       17 &       26 &      344 &      113 &       39 &       56 &

68 &        0 &        7 &       11 &       19 &     4299\\

         1 &        7 &       15 &       17 &       26 &     2920 &      129 &       53 &       48 &

105 &        3 &        9 &       14 &       19 &     8175\\

         1 &        8 &       10 &       18 &       23 &     1064 &      107 &       56 &       99 &

140 &        1 &        7 &       10 &       18 &     7496\\

         1 &        8 &       13 &       19 &       33 &     3225 &      180 &       74 &       92 &

68 &        3 &        9 &       14 &       26 &     9631\\

         1 &        8 &       13 &       22 &       31 &     1228 &      102 &       22 &       51 &

96 &        2 &        9 &       13 &       23 &     6446\\

         1 &        8 &       13 &       22 &       31 &     1535 &      144 &       44 &      102 &

30 &        2 &        9 &       14 &       27 &     6446\\

         1 &        8 &       13 &       22 &       31 &     1842 &      198 &       88 &       51 &

24 &        2 &        9 &       17 &       27 &     6446\\

         1 &        8 &       16 &       21 &       39 &     3873 &       60 &       63 &       22 &

16 &        5 &       11 &       18 &       34 &     6454\\

         1 &       15 &       20 &       24 &       36 &      986 &       29 &        6 &       22 &

27 &        6 &       16 &       19 &       28 &     3058\\

         2 &        4 &        8 &       21 &       35 &      370 &      150 &       30 &       18 &

64 &        1 &        5 &       12 &       22 &     4218\\

         2 &        5 &        9 &       20 &       34 &      387 &       30 &       79 &       56 &

21 &        2 &        4 &       12 &       28 &     3483\\

         2 &        5 &       10 &       14 &       22 &       96 &        7 &       19 &       15 &

6 &        2 &        4 &       10 &       18 &      768\\

         2 &        5 &       11 &       15 &       23 &      908 &      590 &      134 &       51 &

80 &        1 &        7 &       12 &       18 &     8853\\

         2 &        6 &       10 &       15 &       35 &     1225 &       46 &      120 &      164 &

46 &        3 &        5 &       10 &       28 &     8008\\

         2 &        6 &       10 &       15 &       35 &     2058 &      184 &       80 &       41 &

39 &        3 &        7 &       12 &       28 &     8008\\

         2 &        6 &       14 &       23 &       32 &      413 &       20 &       28 &       70 &

18 &        3 &        6 &       14 &       28 &     3540\\

         2 &        6 &       14 &       23 &       32 &      590 &       50 &        7 &       22 &

45 &        3 &       10 &       14 &       24 &     3540\\

         2 &        7 &       14 &       24 &       40 &      564 &       60 &      147 &      131 &

54 &        2 &        6 &       16 &       33 &     8930\\

         2 &        8 &       13 &       20 &       32 &     1036 &      183 &       24 &       58 &

72 &        3 &       10 &       14 &       24 &     7326\\

         2 &        9 &       14 &       20 &       50 &      552 &       44 &      118 &       50 &

16 &        3 &        8 &       16 &       42 &     4968\\

         2 &       10 &       14 &       19 &       22 &      258 &       81 &       52 &        7 &

33 &        2 &       10 &       16 &       18 &     2924\\

        2 &       10 &       16 &       21 &       26 &      198 &       28 &       10 &        3 &

6 &        4 &       12 &       18 &       22 &     1064\\

         2 &       10 &       16 &       23 &       35 &      654 &      258 &       63 &       57 &

62 &        2 &       12 &       18 &       28 &     8393\\

         2 &       10 &       16 &       23 &       35 &     1962 &       86 &       42 &       45 &

54 &        6 &       12 &       18 &       28 &     8393\\

         2 &       10 &       20 &       23 &       34 &     2575 &      133 &       47 &       30 &

52 &        6 &       15 &       20 &       28 &     9888\\

         3 &        6 &        9 &       16 &       28 &      387 &      254 &       70 &       92 &

128 &        2 &        6 &       10 &       19 &     8386\\

         3 &        6 &        9 &       16 &       28 &      516 &      381 &      140 &       23 &

104 &        2 &        6 &       13 &       19 &     8386\\

         3 &        6 &       15 &       20 &       33 &      532 &       82 &       86 &       33 &

112 &        3 &        7 &       16 &       22 &     7753\\

         3 &        6 &       15 &       20 &       33 &      798 &      164 &       43 &       24 &

98 &        3 &       10 &       16 &       22 &     7753\\

         3 &        7 &       12 &       19 &       35 &      450 &      148 &       20 &       51 &

45 &        3 &        9 &       13 &       26 &     5185\\

         3 &       12 &       16 &       23 &       31 &      411 &       51 &       13 &        3 &

46 &        6 &       13 &       19 &       22 &     3563\\

         3 &       12 &       17 &       19 &       30 &      410 &      215 &       60 &       74 &

50 &        3 &       13 &       16 &       25 &     7750\\

         4 &        7 &       12 &       19 &       40 &       78 &       51 &       40 &       21 &

31 &        3 &        7 &       14 &       26 &     2810\\

         4 &        8 &       15 &       20 &       25 &      230 &      170 &      189 &      101 &

38 &        3 &        8 &       16 &       22 &     8098\\

         4 &        9 &       14 &       20 &       48 &      179 &       70 &       12 &       81 &

9 &        4 &       10 &       13 &       42 &     3582\\

         4 &        9 &       16 &       22 &       39 &      129 &      136 &       22 &       18 &

134 &        3 &       11 &       16 &       22 &     7742\\

         5 &        8 &       13 &       22 &       31 &      379 &       85 &       55 &       30 &

80 &        5 &        9 &       16 &       23 &     6446\\

         5 &       10 &       14 &       17 &       42 &      385 &      156 &       66 &      163 &

27 &        5 &       10 &       13 &       36 &     8328\\

         5 &       11 &       19 &       22 &       35 &      344 &      249 &       24 &       75 &

21 &        5 &       15 &       18 &       31 &     7313\\

\hline
\end{tabular}$$}
\end{table}

\begin{proof}From Theorem \ref{th:chongyao}, it is not difficult to find that for $s=5$ the complete $r$-partite graphs
$K_{p_{1},p_{2},\ldots,p_{r}}=K_{a_{1}\cdot p_{1},a_{2}\cdot
p_{2},\ldots,a_{s}\cdot p_{s}}$ is distance integral if and only if integers
$p_{i}(>0)$, $a_{i}(>0)$ and $\mu_{i}(i=1,2,3,4,5)$ satisfy the following equations.
$$ a_{1}=\frac{(\mu_{1}-p_{1}+2)(\mu_{2}-p_{1}+2)(\mu_{3}-p_{1}+2)(\mu_{4}-p_{1}+2)(\mu_{5}-p_{1}+2)}{p_{1}(p_{2}-p_{1})(p_{3}-p_{1})(p_{4}-p_{1})(p_{5}-p_{1})},$$
$$ a_{2}=\frac{(\mu_{1}-p_{2}+2)(\mu_{2}-p_{2}+2)(\mu_{3}-p_{2}+2)(\mu_{4}-p_{2}+2)(\mu_{5}-p_{2}+2)}{p_{2}(p_{1}-p_{2})(p_{3}-p_{2})(p_{4}-p_{2})(p_{5}-p_{2})},$$
$$ a_{3}=\frac{(\mu_{1}-p_{3}+2)(\mu_{2}-p_{3}+2)(\mu_{3}-p_{3}+2)(\mu_{4}-p_{3}+2)(\mu_{5}-p_{3}+2)}{p_{3}(p_{1}-p_{3})(p_{2}-p_{3})(p_{4}-p_{3})(p_{5}-p_{3})},$$
$$ a_{4}=\frac{(\mu_{1}-p_{4}+2)(\mu_{2}-p_{4}+2)(\mu_{3}-p_{4}+2)(\mu_{4}-p_{4}+2)(\mu_{5}-p_{4}+2)}{p_{4}(p_{1}-p_{4})(p_{2}-p_{4})(p_{3}-p_{4})(p_{5}-p_{4})},$$
$$ a_{5}=\frac{(\mu_{1}-p_{5}+2)(\mu_{2}-p_{5}+2)(\mu_{3}-p_{5}+2)(\mu_{4}-p_{5}+2)(\mu_{5}-p_{5}+2)}{p_{5}(p_{1}-p_{5})(p_{2}-p_{5})(p_{3}-p_{5})(p_{4}-p_{5})},$$

From Theorem \ref{th:newclass}, we need only consider the case $GCD(p_1,p_2,\ldots,p_s)=1$. Hence, by using a computer search, we have found 37 integral solutions
listed in Table 1, where $1\leq p_{1}\leq 7,$ $p_{1}+1\leq p_{2}\leq
15,$ $p_{2}+1\leq p_{3}\leq 20,$ $p_{3}+1\leq p_{4}\leq 24,$
$p_{4}+1\leq p_{5}\leq 50,$ $p_{1}-2< \mu_{1}<p_{2}-2$, $p_{2}-2<
\mu_{2}<p_{3}-2$, $p_{3}-2< \mu_{3}<p_{4}-2$, $p_{4}-2<
\mu_{4}<p_{5}-2$, $ \mu_{5}<10000$, and
$n=\sum_{i=1}^{5}a_{i}p_{i}$.

By Theorem \ref{th:newclass} it follows that these graphs
$K_{a_{1}\cdot p_{1}q,a_{2}\cdot p_{2}q,a_{3}\cdot p_{3}q,a_{4}\cdot
p_{4}q,a_{5}\cdot p_{5}q}$ are distance integral for any positive
integer $q$.
\end{proof}

According to Theorem \ref{th:newclass} and Theorem \ref{th:kuozhan}, we can obtain the following
theorem.

\noindent \begin{theorem}\label{th:xinlei} For $s=5$, let
$p_{i}(>0)$, $a_{i}(>0)$ and $\mu_{i}(>0)$ $(i=1,2,3,4,5)$ be those
of Theorem \ref{th:chongyao}. Then for any positive integer $q$, the
complete multipartite graphs $K_{a_{1}\cdot p_{1}q,a_{2}\cdot
p_{2}q,\ldots,a_{5}\cdot p_{5}q}$ with
$qn(=q\sum_{i=1}^{5}a_{i}p_{i})$ vertices are distance integral if
$p_{1}=1$, $p_{2}=4$, $p_{3}=8,$ $p_{4}=18$, $p_{5}=31$,
$\mu_{1}=1$, $\mu_{2}=5$, $\mu_{3}=14$, $\mu_{4}=20$,
$\mu_{5}=9298+13236132t$, $a_{1}=4671576t+3282,$
$a_{2}=472719t+332$, $a_{3}=143871t+101$, $a_{4}=31372t+22$
$a_{5}=159936t+112$, and $n=\sum_{i=1}^{5}a_{i}p_{i}=324632t+2622$,
where $t$ is a nonnegative integer.
\end{theorem}

\begin{proof}
The proof directly follows from Theorem \ref{th:newclass} and Theorem \ref{th:kuozhan}.
\end{proof}

\begin{rem}\label{re:s=5} Analogously to Theorem \ref{th:xinlei} it is possible to find conditions for parameters $p_{i}$, $a_i (i=1,2,3,4,5)$ which depend on $t$ for each graph in Table 1. By this procedure we get new classes of distance integral graphs.

\end{rem}

\noindent \begin{theorem}\label{th:s=6} For $s=6$, let $p_1=2$,
$p_2=5$, $p_3=11$, $p_4=18$, $p_5=21$ and $p_6=23$, $\mu_{1}=1$,
$\mu_{2}=7$, $\mu_{3}=12$, $\mu_{4}=18$, $\mu_{5}=20$ and
$\mu_{6}=53979$, $a_{1}=4735$, $a_{2}=2941$, $a_{3}=514$,
$a_{4}=593$, $a_{5}=213$, $a_{6}=391$.  Then for any positive
integer $q$ the graph $K_{a_{1}\cdot p_{1}q,a_{2}\cdot p_{2}q,}$
$_{a_{3}\cdot p_{3}q,a_{4}\cdot p_{4}q,a_{5}\cdot p_{5}q,a_{6}\cdot
p_{6}q}$ on $qn(=q\sum_{i=1}^{6}a_{i}p_{i})$ vertices is distance
integral.
\end{theorem}

\begin{proof}From Theorem \ref{th:chongyao}, it is not difficult to find that for $s=6$ the complete $r$-partite graphs
$K_{p_{1},p_{2},\ldots,p_{r}}=K_{a_{1}\cdot p_{1},a_{2}\cdot
p_{2},\ldots,a_{s}\cdot p_{s}}$ is distance integral if and only if integers
$p_{i}(>0)$, $a_{i}(>0)$ and $\mu_{i}(i=1,2,3,4,5,6)$ satisfy the following equations.

$$ a_{1}=\frac{(\mu_{1}-p_{1}+2)(\mu_{2}-p_{1}+2)(\mu_{3}-p_{1}+2)(\mu_{4}-p_{1}+2)(\mu_{5}-p_{1}+2)(\mu_{6}-p_{1}+2)}{p_{1}(p_{2}-p_{1})(p_{3}-p_{1})(p_{4}-p_{1})(p_{5}-p_{1})(p_{6}-p_{1})},$$
$$ a_{2}=\frac{(\mu_{1}-p_{2}+2)(\mu_{2}-p_{2}+2)(\mu_{3}-p_{2}+2)(\mu_{4}-p_{2}+2)(\mu_{5}-p_{2}+2)(\mu_{6}-p_{2}+2)}{p_{2}(p_{1}-p_{2})(p_{3}-p_{2})(p_{4}-p_{2})(p_{5}-p_{2})(p_{6}-p_{2})},$$
$$ a_{3}=\frac{(\mu_{1}-p_{3}+2)(\mu_{2}-p_{3}+2)(\mu_{3}-p_{3}+2)(\mu_{4}-p_{3}+2)(\mu_{5}-p_{3}+2)(\mu_{6}-p_{3}+2)}{p_{3}(p_{1}-p_{3})(p_{2}-p_{3})(p_{4}-p_{3})(p_{5}-p_{3})(p_{6}-p_{3})},$$
$$ a_{4}=\frac{(\mu_{1}-p_{4}+2)(\mu_{2}-p_{4}+2)(\mu_{3}-p_{4}+2)(\mu_{4}-p_{4}+2)(\mu_{5}-p_{4}+2)(\mu_{6}-p_{4}+2)}{p_{4}(p_{1}-p_{4})(p_{2}-p_{4})(p_{3}-p_{4})(p_{5}-p_{4})(p_{6}-p_{4})},$$
$$ a_{5}=\frac{(\mu_{1}-p_{5}+2)(\mu_{2}-p_{5}+2)(\mu_{3}-p_{5}+2)(\mu_{4}-p_{5}+2)(\mu_{5}-p_{5}+2)(\mu_{6}-p_{5}+2)}{p_{5}(p_{1}-p_{5})(p_{2}-p_{5})(p_{3}-p_{5})(p_{4}-p_{5})(p_{6}-p_{5})},$$
$$ a_{6}=\frac{(\mu_{1}-p_{6}+2)(\mu_{2}-p_{6}+2)(\mu_{3}-p_{6}+2)(\mu_{4}-p_{6}+2)(\mu_{5}-p_{6}+2)(\mu_{6}-p_{6}+2)}{p_{6}(p_{1}-p_{6})(p_{2}-p_{6})(p_{3}-p_{6})(p_{4}-p_{6})(p_{5}-p_{6})},$$

From Theorem \ref{th:newclass}, we need only consider the case
$GCD(p_1,p_2,\ldots,p_s)=1$. Hence, by using a computer search, we
found this result, where $1\leq p_{1}\leq 5,$ $p_{1}+1\leq p_{2}\leq
p_{1}+10,$ $p_{2}+1\leq p_{3}\leq p_{2}+10,$ $p_{3}+1\leq p_{4}\leq
p_{3}+10,$ $p_{4}+1\leq p_{5}\leq p_{4}+10,$ $p_{5}+1\leq p_{6}\leq
p_{5}+20,$ $p_{1}-2< \mu_{1}<p_{2}-2$, $p_{2}-2< \mu_{2}<p_{3}-2$,
$p_{3}-2< \mu_{3}<p_{4}-2$, $p_{4}-2< \mu_{4}<p_{5}-2$, $ p_{5}-2<
\mu_{5}< p_{6}-2$, $p_{6}-2< \mu_{6}<200000$ and
$n=\sum_{i=1}^{6}a_{i}p_{i}$.

By Theorem \ref{th:newclass} it follows that these graphs
$K_{a_{1}\cdot p_{1}q,a_{2}\cdot p_{2}q,a_{3}\cdot p_{3}q,a_{4}\cdot
p_{4}q,a_{5}\cdot p_{5}q,a_{6}\cdot p_{6}q}$ are distance integral
for any positive integer $q$.\end{proof}

According to Theorem \ref{th:newclass} and Theorem \ref{th:kuozhan}, we can obtain the following
theorem.

\noindent \begin{theorem}\label{th:s6xinlei} For $s=6$, let
$p_{i}(>0)$, $a_{i}(>0)$ and $\mu_{i}(>0)$ $(i=1,2,3,4,5,6)$ be
those of Theorem \ref{th:chongyao}. Then for any positive integer
$q$, the complete multipartite graphs $K_{a_{1}\cdot
p_{1}q,a_{2}\cdot p_{2}q,a_{3}\cdot p_{3}q,}$ $_{ \ldots,a_{5}\cdot
p_{5}q,a_{6}\cdot p_{6}q}$ with $qn(=q\sum_{i=1}^{6}a_{i}p_{i})$
vertices are distance integral if $p_{1}=2$, $p_{2}=5$, $p_{3}=11,$
$p_{4}=18$, $p_{5}=21$, $p_{6}=23$, $\mu_{1}=1$, $\mu_{2}=7$,
$\mu_{3}=12$, $\mu_{4}=18$, $\mu_{5}=20$, $\mu_{6}=9598038t+53979$,
$a_{1}=418600t+4735,$ $a_{2}=260015t+2941$, $a_{3}=45448t+514$,
$a_{4}=52440t+593$, $a_{5}=18837t+213$, $a_{6}=34580t+391$,
$n=\sum_{i=1}^{6}a_{i}p_{i}=829920t+9387$, where $t$ is a
nonnegative integer.
\end{theorem}

\begin{proof}
The proof directly follows from Theorem \ref{th:newclass} and Theorem \ref{th:kuozhan}.
\end{proof}

%\section*{Acknowledgements}
%The authors would like to express their sincere gratitude to  the
%anonymous referees for their comments and remarks, which improved
%the presentation of this paper.

\end{CJK*}

\begin{thebibliography}{99}



\bibitem{AH} AOUCHICHE M, HANSEN P.  Two Laplacians for the distance matrix of a
graph[J]. Linear Algebra Appl, 2013, 439: 21--33.


%\bibitem{BoEr} P. Borwein, T. Erd\'{e}lyi,  Polynomials and polynomial inequalities, Springer, New York, Berlin, Heidelberg, 1995 106--107.

\bibitem{BNP} BOSE S S, NATH M, PAUL S. Distance spectral radius of graphs with
$r$ pendent vertices[J]. Linear Algebra Appl, 2011, 435: 2828--2836.

\bibitem{BuHa} BUCKLEY F, HARARY F. Distance in Graphs[M]. Addison-Wesley, Redwood City, CA, 1990.



\bibitem{Das} DAS K C. On the largest eigenvalue of the distance matrix of a
bipartite graph[J]. MATCH Commun Math Comput Chem, 2009, 62: 667--672.



\bibitem{HP2} H\'{I}C P, POKORN\'{Y} M. New Classes of Integral Complete $n$-partite
Graphs[J]. Advances and Applications in Discrete Mathematics, 2011, 7(2): 83--94.

\bibitem{InGu} INDULAL G, GUTMAN I. On the distance spectra of some graphs[J]. Math
Commun, 2008, 13: 123--131.

%\bibitem{InGV} G. Indulal, I. Gutman, A. Vijayakumar, On distance energy of graphs, MATCH Commun. Math. Comput. Chem. 60 (2008) 461--472.

\bibitem{LHWS} LIN H Q, HONG Y, WANG J F, SHU J L.  On the distance spectrum of graphs[J]. Linear Algebra Appl, 2013, 439: 1662--1669.

\bibitem{LZ} LU S F, ZHAO H X, Signless Laplacian characteristic polynomials of complete multipartite graphs[J].  Chinese Quarterly Journal of Mathematics, 2012, 27(1): 36--40.

\bibitem{Mer} MERRIS R.  The distance spectrum of a tree[J]. J. Graph Theory, 1990, 14: 365--369.

\bibitem{SMHP} STEVANOVI\'{C} D, MILOSEVI\'{C} M, H\'{I}C P, POKORN\'{Y} M. Proof of a conjecture
on distance energy of complete multipartite graphs[J]. MATCH
Commun Math Comput Chem, 2013, 70: 157-162.

\bibitem{StIl} STEVANOVI\'{C} D, ILI\'{C} A.  Distance spectral radius of trees with fixed maximum degree[J]. Electron
J Linear Algebra, 2010, 20: 168--179.

\bibitem{StIn} STEVANOVI\'{C} D, INDULAL G.  The distance spectrum and energy of the compositions
of regular graphs[J]. Appl Math Lett, 2009, 22: 1136--1140.

\bibitem{T} TAN S W, On the Laplacian spectral radius of trees[J], Chinese Quarterly Journal of Mathematics, 2010, 25(4): 615--625.

\bibitem{WW}WANG L G, WANG Q. Integral complete multipartite graphs $K_{a_{1}\cdot p_{1},a_{2}\cdot
p_{2},\ldots,a_{s}\cdot p_{s}}$ with s=5,6[J]. Discrete Math, 2010, 310: 812--818.

\bibitem{WaZhou}WANG Y N, ZHOU B.  On distance spectral radius of graphs[J]. Linear Algebra
Appl, 2013, 438: 3490--3503.


\bibitem{YW} YANG R S, WANG L G. Distance integral complete $r$-partite graphs[J]. Filomat, accepted, 2014.

\bibitem{Zhou} ZHOU B.  On the largest eigenvalue of the distance matrix of a
tree[J]. MATCH Commun Math Comput Chem, 2007, 58: 657--662.

\bibitem{ZhIl} ZHOU B, ILI{\'{C}} A. On distance spectral radius and distance energy
of graphs[J]. MATCH Commun Math Comput Chem, 2010, 64: 261--280.

\bibitem{Zho} ZHOU H Q. The rank of integral circulant graphs[J]. Chinese Quarterly Journal of Mathematics, 2014, 29(1):116--124.

\end{thebibliography}
\end{document}